\documentclass[11pt,oneside,reqno]{amsart}

\usepackage{amsmath,amsthm} 
\usepackage{mathtools}
\usepackage{enumerate}
\usepackage{xfrac}          
\numberwithin{equation}{section}
\theoremstyle{definition}
\newtheorem{Definition}{Definition}[section]
\newtheorem{Example}[Definition]{Example}
\newtheorem{Remark}[Definition]{Remark}

\theoremstyle{plain}
\newtheorem{Theorem}[Definition]{Theorem}

\newtheorem{Proposition}[Definition]{Proposition}
\newtheorem{Corollary}[Definition]{Corollary}
\newtheorem{Lemma}[Definition]{Lemma}

\usepackage{amssymb}        
\usepackage{mathrsfs}       
\usepackage{eucal}          
\usepackage{stmaryrd}       

\usepackage[margin=1in]{geometry}

\usepackage{hyperref}

\usepackage[usenames,dvipsnames]{xcolor}
\usepackage{tikz}
\usetikzlibrary{arrows, matrix}
\usepackage{tikz-cd}
\usepackage{subfig}         
\usepackage{arydshln}       

\newcommand{\al}{\alpha}

\newcommand{\ga}{\gamma}
\newcommand{\Ga}{\Gamma}

\newcommand{\La}{\Lambda}
\newcommand{\si}{\sigma}

\newcommand{\N}{\mathbb{N}}
\newcommand{\Z}{\mathbb{Z}}

\newcommand{\C}{\mathbb{C}}
\newcommand{\K}{\Bbbk}

\newcommand{\Fgl}{\mathfrak{gl}}
\newcommand{\Fsl}{\mathfrak{sl}}

\newcommand{\Fm}{\mathfrak{m}}

\newcommand{\CV}{\mathcal{V}}

\newcommand{\SA}{\mathscr{A}}

\newcommand{\SK}{\mathscr{K}}
\newcommand{\SL}{\mathscr{L}}
\newcommand{\SM}{\mathscr{M}}

\newcommand{\SU}{\mathscr{U}}
\newcommand{\SX}{\mathscr{X}}

\newcommand{\op}{\operatorname}

\DeclareMathOperator{\Aut}{Aut}

\DeclareMathOperator{\Frac}{Frac}
\DeclareMathOperator{\Gal}{Gal}

\DeclareMathOperator{\Hom}{Hom}

\DeclareMathOperator{\sgn}{sgn}

\DeclareMathOperator{\Supp}{Supp}

\DeclareMathOperator{\GL}{GL}

\newcommand{\iv}[2]{\llbracket #1,#2 \rrbracket}

\renewcommand{\hat}{\widehat}
\renewcommand{\tilde}{\widetilde}

\def\Tiny{\fontsize{4pt}{4pt}\selectfont}
\newcommand*{\eqdef}{\ensuremath{\overset{\mathclap{\text{\Tiny def}}}{=}}}
\renewcommand{\subset}{\subseteq}

\title{Principal Galois orders and Gelfand-Zeitlin modules}
\author{Jonas T. Hartwig}

\address{Department of Mathematics, Iowa State University, Ames, IA-50011, USA}
\email{jth@iastate.edu}
\urladdr{http://jth.pw}

\begin{document}

\maketitle

\begin{abstract}
We show that the ring of invariants in a skew monoid ring contains a so called standard Galois order. Any Galois ring contained in the standard Galois order is automatically itself a Galois order and we call such rings principal Galois orders. We give two applications. First, we obtain a simple sufficient criterion for a Galois ring to be a Galois order and hence for its Gelfand-Zeitlin subalgebra to be maximal commutative. Second, generalizing a recent result by Early-Mazorchuk-Vishnyakova, we construct canonical simple Gelfand-Zeitlin modules over any principal Galois order.

As an example, we introduce the notion of a rational Galois order, attached an arbitrary finite reflection group and a set of rational difference operators, and show that they are principal Galois orders. Building on results by Futorny-Molev-Ovsienko, we show that parabolic subalgebras of finite W-algebras are rational Galois orders. Similarly we show that Mazorchuk's orthogonal Gelfand-Zeitlin algebras of type $A$, and their parabolic subalgebras, are rational Galois orders. Consequently we produce canonical simple Gelfand-Zeitlin modules for these algebras and prove that their Gelfand-Zeitlin subalgebras are maximal commutative.

Lastly, we show that quantum OGZ algebras, previously defined by the author, and their parabolic subalgebras, are principal Galois orders. This in particular proves the long-standing Mazorchuk-Turowska conjecture that, if $q$ is not a root of unity, the Gelfand-Zeitlin subalgebra of $U_q(\Fgl_n)$ is maximal commutative and that its Gelfand-Zeitlin fibers are non-empty and (by Futorny-Ovsienko theory) finite.
\end{abstract}


\subsubsection*{Notation}

The multiplicative identity element in a monoid or ring $S$ is denoted $1_S$. All rings are assumed to have an identity. $\Frac A$ denotes the field of fractions of an Ore domain $A$.
For a group $G$ acting by automorphisms on a ring $A$ we let $A^G=\{a\in A\mid \forall g\in G:\; g(a)=a\}$ denote the subring of $G$-invariants.

\section{Introduction}

\subsection{Galois orders}
In this paper we study Galois rings and orders, which are certain noncommutative rings defined and studied by Futorny and Ovsienko \cite{FutOvs2010,FutOvs2014}. This theory has its origin in the Gelfand-Zeitlin bases for finite-dimensional simple modules over classical Lie algebras and groups \cite{GelZei1950a,GelZei1950b,Zhelobenko1973} followed by the foundational paper \cite{DroFutOvs1994} where several notions such as Gelfand-Zeitlin modules where defined in a very general setting. 

Gelfand-Zeitlin modules have been studied from many different points of view \cite{Mazorchuk1999,MazTur2000,MazPonTur2003,Ovsienko2003,FutMolOvs2010a,FutMolOvs2010,KosWal2006a,KosWal2006b,ColEve2014}. Classification of simple Gelfand-Zeitlin modules for $U(\Fgl_n)$ is still out of reach although tremendous progress has been made in the last few years, progressing from generic modules in \cite{DroFutOvs1994} to the construction of increasingly singular modules in \cite{FutGraRam2014,FutGraRam2015,FutGraRam2016,FutGraRam2017}. In \cite{GomRam2016} a necessary condition for certain fully supported singular modules to be simple was proved, and simple subquotients described in some cases.
 Different realizations of special singular Gelfand-Zeitlin modules were described in \cite{Zadunaisky2017,Vishnyakova2017a,Vishnyakova2017b}.
This development culminated in the construction by Ram\'{i}rez and Zadunaisky \cite{RamZad2017} of fully supported Gelfand-Zeitlin modules associated to \emph{any} character in a uniform way using divided difference operators. Moreover they were able to calculate the dimensions of all generalized weight spaces in each module. It is conjectured that any simple Gelfand-Zeitlin module over $U(\Fgl_n)$ will be a subquotient of one of those.
In the latest development \cite{EarMazVis2017}, the authors moved beyond $U(\Fgl_n)$ and constructed canonical Gelfand-Zeitlin modules for so called \emph{orthogonal Gelfand-Zeitlin (OGZ) algebras} \cite{Mazorchuk1999}, and constructed bases for some of them.

In this paper we return to the general framework in \cite{FutOvs2010,FutOvs2014} and propose a simplified setup for Galois orders involving few, but natural, assumptions and then develop our results from those. All important examples of Galois rings and orders in the literature satisfy the conditions including $U(\Fgl_n)$, $U_q(\Fgl_n)$, OGZ algebras and finite W-algebras. In addition our setup sometimes enables simpler proofs over \cite{FutOvs2010}. The main generalization is that we do not require the monoid $\SM$ to be a group in any of the results. This leads to various ``parabolic'' examples. Moreover we do not require a ground field until we get to Gelfand-Zeitlin modules in Section \ref{sec:GZ-modules}.

The first main result of this paper provides a sufficient condition for a ring to be a Galois order and simultaneously provides canonical simple Gelfand-Zeitlin modules. This generalizes the construction of modules from \cite{EarMazVis2017}, and can also be viewed as a partial generalization of the statement about non-empty fibers in \cite{FutOvs2014}.

\begin{Theorem}\label{thm:I}
With assumptions as in Section \ref{sec:invariants}, let $\SU$ be a subring of $\SK$ generated by $\Ga\cup \SX$ where $\cup_{X\in\SX}\Supp_\SM(X)$ generates $\SM$ as a monoid, and $X(\ga)\in\Ga$ for all $X\in\SX$ and $\ga\in \Ga$. Then:
\begin{enumerate}[{\rm (i)}]
\item $\SU$ is a Galois $\Ga$-order in $\SK$,
\item $\Ga$ is maximal commutative in $\SU$.
\end{enumerate}
If in addition $\La$ is finitely generated over an algebraically closed field $\K$ of characteristic zero, and $\SM$ and $G$ act by $\K$-automorphisms on $\La$, then:
\begin{enumerate}[{\rm (i)}]
\setcounter{enumi}{2}
\item For every character $\xi:\Ga\to\K$, there exists a canonical simple right Gelfand-Zeitlin $\SU$-module $V(\xi)$ with $V(\xi)_\xi\neq 0$. Moreover $V(\xi)$ is realized as a quotient of a cyclic submodule of $\Ga^\ast=\Hom_\K(\Ga,\K)$.
\end{enumerate}
\end{Theorem}

We call Galois orders satisfying the hypothesis of Theorem \ref{thm:I} \emph{principal Galois orders}.
Replacing the condition $X(\ga)\in\Ga$ by $X^\dagger(\ga)\in\Ga$ where $\dagger$ is a certain anti-isomorphism (see Remark \ref{rem:dagger}) gives the definition of \emph{co-principal Galois orders}, and for those one obtains left modules instead of right modules. 

Further exposing the connection to invariant theory already manifest in \cite{FutOvs2010,FutOvs2014,Zadunaisky2017,Vishnyakova2017a,Vishnyakova2017b,RamZad2017}, our second main result is the construction of a new class of Galois orders, called \emph{rational Galois orders}, see Definition \ref{def:rational-GO}. They are attached to an arbitrary finite (pseudo-)reflection group and a set of difference operators with rational function coefficients. They are tailored to naturally satisfy the hypotheses of Theorem \ref{thm:I}, using the main result from \cite{Terao1989} about relative invariants. In the sections that follow we apply these results to parabolic subalgebras of finite W-algebras, OGZ algebras and quantum OGZ algebras. The terminology ``rational Galois order'' is introduced with a different meaning in an unfinished manuscript by Futorny and Ovsienko.

\subsection{Finite W-algebras}

Finite W-algebras have many realizations;
as generalizations of enveloping algebras depending on a nilpotent element \cite{Premet2002};
as quantizations of Slodowy slices \cite{GanGin2002};
as truncated shifted Yangians \cite{RagSor1999,BruKle2006};
and as Galois orders \cite{FutMolOvs2010}.
The family of finite W-algebras $W(\pi)$ of type $A$ includes level $p$ Yangians $Y_p(\Fgl_n)$ and the enveloping algebra $U(\Fgl_n)$ as special cases.
 
In Section \ref{sec:finite-W}, we prove that finite W-algebras, and their parabolic subalgebras obtained by removing a subset of the ``negative simple root vectors'', are examples of rational Galois orders. Then we apply Theorem \ref{thm:I}. As a consequence we obtain our third main result. 

\begin{Theorem} \label{thm:intro-finite-W}
Let $\K$ be an algebraically closed field of characteristic zero. Let $J$ be a subset of $\{1,2,\ldots,n-1\}$, and let $W_J(\pi)$ be the parabolic subalgebra of the finite W-algebra of type $A$,  generated by the coefficients of the polynomials $A_i(u)$, $B^-_j(u)$ and $B^+_k(u)$, where $1\le i\le n$, $j\in J$ and $1\le k\le n-1$. Then:
\begin{enumerate}[{\rm (i)}]
\item $W_J(\pi)$ is isomorphic to a rational Galois order;
\item For every character $\xi:\Ga\to\K^\times$, there exists a canonical simple left Gelfand-Zeitlin $W_J(\pi)$-module $M$ with $M_\xi\neq 0$.
\end{enumerate}
\end{Theorem}

In particular part (ii) says that the Gelfand-Zeitlin fibers are non-empty.
This, and the fact $W_J(\pi)$ is a Galois order, was already known when $J=\{1,2,\ldots,n\}$ by \cite{FutMolOvs2010}. However, their method cannot be applied for the proper parabolic subalgebras, because the corresponding monoid $\SM$ is not a group. We prove the result by realizing them as rational Galois orders and then applying Theorem \ref{thm:I} which holds any rational Galois order.

\subsection{Orthogonal Gelfand-Zeitlin algebras and their quantizations} 

In \cite{Mazorchuk1999}, a family of algebras of linear operators were defined called \emph{orthogonal Gelfand-Zeitlin (OGZ) algebras}, denoted $U(\boldsymbol{r})$. They include (algebras isomorphic to) $U(\mathfrak{gl}_n)$ and extended Heisenberg algebras as special cases.
\emph{Quantum OGZ algebras} $U_q(\boldsymbol{r})$ were defined by the author in \cite{Hartwig2017}. It was shown in \cite{Hartwig2017} that $U_q(\boldsymbol{r})$ are Galois rings (which is weaker than being a Galois order). By the same methods it can be seen that $U(\boldsymbol{r})$ are Galois rings.
Among OGZ algebras, only $U(\Fgl_n)=U(1,2,\ldots,n)$ and the enveloping algebra of the trivially extended Heisenberg Lie algebra $U(\Fsl_3^+ \times \K t)\cong U(1,1)$ are known to be Galois orders, the former shown in \cite{FutOvs2010} and the latter follows by results in \cite{FutOvs2010} since they are generalized Weyl algebras. Similarly, only $U_q(1,1)$ is known to be a Galois order, again because it is a generalized Weyl algebra. 
Moreover, the Gr\"{o}bner basis methods in \cite{FutOvs2010,FutMolOvs2010}, do not work in the quantum case, essentially because the associated graded algebra is not commutative. Although it is likely their method might work for $U(\boldsymbol{r})$, our method of realizing the algebras as rational Galois orders works uniformly in all cases, including the parabolic generalizations.

In Sections \ref{sec:OGZ} and \ref{sec:qOGZ} we prove our fourth main result.

\begin{Theorem}\label{thm:main-(q)OGZ}
Let $\K$ be an algebraically closed field of characteristic zero.
Let $J\subset\{1,2,\ldots,n-1\}$ and $U_J$ be $U_J(\boldsymbol{r})$ or $U_q(\boldsymbol{r};J)$ i.e. a parabolic (quantum) OGZ algebra over $\K$ and $\Ga$ be its Gelfand-Zeitlin subalgebra. Then:
\begin{enumerate}[{\rm (i)}]
\item $U$ is a Galois $\Ga$-order;
\item $\Ga$ is maximal commutative in $U$;
\item For any character $\xi:\Ga\to\K$, there exists a canonical simple left Gelfand-Zeitlin $U$-module $M$ with $M_\xi\neq 0$;
\item If $J=\{1,2,\ldots,n\}$, then for any $\xi\in\hat{\Ga}$ there are only finitely many isomorphism classes of simple Gelfand-Zeitlin $U$-modules $M$ with $M_\xi\neq 0$.
\end{enumerate}
\end{Theorem}

As before, we prove this for $U_J(\boldsymbol{r})$ by realizing it as a rational Galois order. However, for the quantum case we show that it satisfies a criterion involving the quantum Vandermonde determinant to prove that it is a co-principal Galois order and then apply Theorem \ref{thm:I}.
To prove part (iv) of this theorem we apply the main result of \cite{FutOvs2014}. 

Lastly, we prove a conjecture of Mazorchuk and Turowska \cite{MazTur2000} stating that the Gelfand-Zeitlin subalgebra of $U_q(\Fgl_n)$ is maximal commutative. By results in \cite{FutOvs2010,FutOvs2014}, this implies that $U_q(\Fgl_n)$ is a Galois order and that the Gelfand-Zeitlin fibers are nonempty and finite.
Since $U_q(\Fgl_n)\cong U_q(1,2,\ldots,n)$ when $\K=\C$, as shown in \cite{FutHar2014,Hartwig2017}, the proof of the conjecture is a direct consequence of Theorem \ref{thm:main-(q)OGZ}.

\begin{Theorem}
Let $n>0$ and $q\in\C\setminus\{0\}$ not a root of unity. Let $U_q(\Fgl_n)$ be the quantized enveloping algebra of $\Fgl_n$ over $\C$. Then the Gelfand-Zeitlin subalgebra $\Ga_q$ of $U_q(\Fgl_n)$ is maximal commutative. Hence $U_q(\Fgl_n)$ is a Galois order with respect to $\Ga_q$. Moreover, for any character $\xi:\Ga_q\to\C$ there exists at least one but only finitely many isomorphism classes of simple Gelfand-Zeitlin $U_q(\Fgl_n)$-module $M$ with $M_\xi\neq 0$.
\end{Theorem}

\section*{Acknowledgements}
The author thanks Mark Colarusso, Ian Musson, Akaki Tikaradze and Tathagata Basak for inspiring and helpful discussions.

\section{Galois orders} \label{sec:Galois-orders}

In this section we develop the theory of Galois orders from first principles. Compared to the treatments \cite{FutOvs2010,FutOvs2014} our scope is modest. We have tried to select a small set of standing assumptions which balances generality and applicability. In particular we assume from the outset that $\La$ (equivalently $\Ga$) is integrally closed and that $\SM$ acts on $\La$ rather than $L$. This ensures that $\Ga$ is always a Harish-Chandra subalgebra of any Galois ring, which in turn means that the theory of Gelfand-Zeitlin modules is well behaved. On the other hand, we do not require that $\SM$ is a group. This allows us to include parabolic subalgebras of finite W-algebras and (quantum) orthogonal Gelfand-Zeitlin algebras as examples of Galois orders. Furthermore, we do not need to assume that $\La$ contains a field, until we get to Gelfand-Zeitlin modules in Section \ref{sec:GZ-modules}.

\subsection{Invariants in skew monoid rings} \label{sec:invariants}
Let $\La$ be an integrally closed domain, $G$ a finite subgroup of $\Aut(\La)$, $\mathscr{M}$ a submonoid of $\Aut(\La)$. We make the following assumptions:
\begin{align}
\text{(\emph{separation})}\quad & \tag{A1} \label{it:G1} (\mathscr{M}\mathscr{M}^{-1})\cap G=1_{\Aut(\La)},\\
\text{(\emph{invariance})}\quad & \tag{A2} \label{it:G2} \forall g\in G\;\forall\mu\in\mathscr{M}:\; {}^g\mu\eqdef g\circ\mu\circ g^{-1}\in\mathscr{M}, \\
\text{(\emph{finiteness})}\quad & \tag{A3} \label{it:G3} \text{$\La$ is noetherian as a module over $\La^G$.}
\end{align}
Put $L=\Frac \La$. Then $\SM$ and $G$ act naturally on $L$ by automorphisms. Let $\mathscr{L}=L\ast\mathscr{M}$ be the skew monoid ring, defined as the free left $L$-module on $\SM$ with multiplication given by  $a_1\mu_1\cdot a_2\mu_2 = \big(a_1\mu_1(a_2)\big) \mu_1\mu_2$ for $a_i\in L,\, \mu_i\in\SM$. 
By \eqref{it:G2}, $G$ acts on $\SL$ by ring automorphisms via $g(a\mu)=g(a)\,{}^g\mu$ for $g\in G,\, a\in L,\, \mu\in\SM$. Let $\Ga=\La^G,\, K=L^G,\, \SK=\SL^G$ be the respective subrings of $G$-invariants. Thus we have the following inclusions:
\begin{equation}
\begin{tikzcd}
\La \arrow[hook]{r}  & L \arrow[hook]{r} & \SL \\ 
\Ga \arrow[hook]{r} \arrow[hook]{u}  & K \arrow[hook]{r} \arrow[hook]{u}& \SK \arrow[hook]{u}
\end{tikzcd}
\end{equation}
\begin{Lemma} \label{lem:first-lemma}
The following statements hold:
\begin{enumerate}[{\rm (i)}]
\item $\La$ is integral over $\Ga$;
\item $K=\Frac\Ga$, $K=L^G$ and $L/K$ is a Galois extension with $\Gal(L/K)=G$;
\item $\Ga$ is integrally closed;
\item $\La$ is the integral closure of $\Ga$ in $L$;
\item $\La$ is a finitely generated $\Ga$-module and a noetherian ring;
\item $\Ga$ is a noetherian ring.
\end{enumerate}
\end{Lemma}

\begin{proof}
(i)(ii)(iii): See \cite[Prop.~3.1]{Broue}.

\smallskip\noindent
(iv): Let $\overline{\Ga}$ be the integral closure of $\Ga$ in $L$. By (i), $\La\subset\overline{\Ga}$. Since $\La$ is integrally closed, $\overline{\Ga}\subset \La$.

\smallskip\noindent
(v): Clear, by \eqref{it:G3}.

\smallskip\noindent
(vi): Follows from (v) and the Eakin-Nagata theorem \cite[Thm.~3.7(i)]{Matsumura}.
\end{proof}

\begin{Example} \label{ex:symmetric-difference-operators-1}
Take $\La=\C[x_1,x_2,\ldots,x_n]$, $G\cong S_n$ acting on $\La$ by permutation of variables, $\SM=\langle \mu_1,\ldots,\mu_n \rangle\cong\N^n$ acting on $\La$ by $\mu_i(x_j)=x_j+\delta_{ij}$. Then \eqref{it:G1}--\eqref{it:G3} hold. In this case $\SK$ can be identified with the ring of symmetric difference operators with rational function coefficients.
\end{Example}

\subsection{Harish-Chandra subrings} \label{sec:HC-subrings}

\begin{Definition}[\cite{DroFutOvs1994}]
A commutative subring $C$ of a ring $\SA$ is a \emph{Harish-Chandra subring (of $\SA$)} if for every $x\in \SA$, the $C$-bimodule $CxC$ is finitely generated as a left and right $C$-module.
\end{Definition}

\begin{Lemma}\label{lem:HC-subrings}
If $C$ is a commutative Harish-Chandra subring of a ring $\SA$, then $C$ is a Harish-Chandra subring of any subring of $\SA$ that contains $C$.
\end{Lemma}

\begin{proof}
Obvious.	
\end{proof}

\begin{Proposition} \label{prp:HC}
$\Ga$ is a Harish-Chandra subring of $\SL$.
\end{Proposition}

\begin{proof}
We follow the proof of \cite[Prop.~5.1]{FutOvs2010}. Since the subset $L\cup\SM$ generates $\SL$ as a ring, $\Ga$ is noetherian and $\Ga(X+Y)\Ga\subset \Ga X+\Ga Y$, $\Ga (XY)\Ga\subset (\Ga X\Ga)(\Ga Y\Ga)$, it suffices to show that $\Ga X \Ga$ is finitely generated as a left and right $\Gamma$-module for all $X\in L\cup\SM$. For $X\in L$ this is trivial since $L$ is commutative. Let $X=\mu\in\SM$. We have
\[\Ga \mu \Ga \subset \Ga \mu(\Ga) \mu \subset \La \mu,\]
since $\Ga\subset\La$ and $\mu(\La)=\La$.  Since $\Lambda \mu\cong\Lambda$ is a finitely generated left $\Gamma$-module and $\Ga$ is noetherian, $\Ga\mu\Ga$ is finitely generated as a left $\Ga$-module. Similarly $\Ga\mu\Ga\subset \mu\mu^{-1}(\Ga)\Ga\subset \mu\La$ shows that $\Ga\mu\Ga$ is finitely generated as a right $\Gamma$-module.
\end{proof}

\subsection{Galois rings} \label{sec:Galois-rings}
We recall the definition of a Galois $\Ga$-ring in $\SK$ from \cite{FutOvs2010}, and give a necessary and sufficient condition for a $\Ga$-subring of $\SU$ to be a Galois $\Ga$-ring in $\SK$, Proposition \ref{prp:Galois-Ring-Equivalent-Condition}. The condition is the same as the one given by Futorny and Ovsienko in \cite[Prop.~4.1(1)]{FutOvs2010}, however there it was stated under different assumptions. And although we follow their proof closely, we have adapted it to our setting and simplified it in several places.

\begin{Definition}[\cite{FutOvs2010}]
 A $\Ga$-subring $\SU\subset \SK$ is a \emph{Galois $\Ga$-ring in $\SK$} if
\begin{equation} \label{eq:Galois-Ring-Condition}
\SU K = \SK = K \SU.
\end{equation}
\end{Definition}

\begin{Lemma} \label{lem:Galois-Ring-Equivalent-Condition}
Let $\mu,\mu_1,\mu_2\in\SM$. Then:
\begin{enumerate}[{\rm (i)}]
\item $K\mu(K) = L^{G_\mu}$, where $G_\mu\eqdef\{g\in G\mid {}^g\mu=\mu\}$ is the $G$-stabilizer of $\mu$;
\item $K[\mu]K=\big\{[a\mu]\mid a\in L^{G_\mu}\big\}$, where $[a\mu]\eqdef\sum_{g\in G/G_\mu} g(a)\,{}^g\mu$, and this is a simple $K$-bimodule;
\item $K[\mu_1]K\cong K[\mu_2]K$ as $K$-bimodules iff $\mu_2={}^g\mu_1$ for some $g\in G$;
\item $K[\mu]\Ga=K[\mu]K=\Ga[\mu]K$, hence $\SU K=K\SU$ for any $\Ga$-subbimodule $\SU\subset\SL$;
\item $\SK = \bigoplus_{\mu\in\SM/G} K[\mu]K$.
\end{enumerate}
\end{Lemma}

\begin{proof}
(i): $g\in\Gal\big(L/K\mu(K)\big)\Leftrightarrow g\circ\mu\big|_K=\mu\big|_K \Leftrightarrow {}^g\mu\big|_K=\mu\big|_K \Leftrightarrow {}^g\mu\mu^{-1}\in G\overset{\eqref{it:G1}}{\Leftrightarrow} g\in G_\mu$.

\smallskip\noindent
(ii): By (i), $K[a\mu]K=\big\{[b\mu]\mid b\in L^{G_\mu}\big\}$ for any nonzero $a\in L^{G_\mu}$.

\smallskip\noindent
(iii): Suppose $\psi:K[\mu_1]K\overset{\sim}{\to} K[\mu_2]K$. Put $L_i=L^{G_{\mu_i}}$ for $i=1,2$. Define $\tilde{\psi}:L_1\to L_2$ by $\psi([a\mu_1])=[\tilde{\psi}(a)\mu_2]$. One checks that $\tilde{\psi}$ is a $K$-isomorphism between the intermediate fields $L_1$ and $L_2$. By Galois theory $\tilde{\psi}=g\big|_{L_1}$ for some $g\in G$. Thus $\psi([a\mu_1])=[g(a)\mu_2]$. Then for $k\in K$: $[{}^g\mu_1(k)\mu_2]=\psi([\mu_1(k)\mu_1])=\psi([\mu_1])k=[\mu_2]k=[\mu_2(k)\mu_2]$ proving $\mu_2\big|_K={}^g\mu_1\big|_K$ hence $\mu_2={}^g\mu_1$. Conversely, if $\mu_2={}^g\mu_1$, the map $[a\mu_1]\mapsto [g(a)\mu_2]$ is an isomorphism $K[\mu_1]K\overset{\sim}{\to} K[\mu_2]K$.

\smallskip\noindent (iv): 
$K[\mu]\Ga=[K\mu(\Ga)\mu]$. By symmetrizing denominators, $K\mu(K)=K\mu(\Ga)$. Then use (i) and (ii). The second equality is analogous.

\smallskip\noindent (v): Let $x\in\SK$ and write  $x=\sum_\mu x_\mu \mu$ for some $x_\mu\in L$. Since $g(x)=x$ for all $g\in G$, $g(x_\mu)=x_{{}^g\mu}$ for all $\mu\in\SM$ and $g\in G$. Grouping terms according to $G$-orbits in $\SM$ this implies that $x=\sum_i [a_i\mu_i]$ for some $\mu_i\in\SM$ from pairwise distinct $G$-orbits in $\SM$, where $a_i=x_{\mu_i}\in L^{G_{\mu_i}}$.
\end{proof}

\begin{Definition} \label{def:Galois-ring}
The \emph{support} of $X=\sum_{\mu\in\SM} x_\mu\mu\in \SL$ is $\Supp_\SM(X)\eqdef\{\mu\in\SM\mid x_\mu\neq 0\}$.
\end{Definition}

\begin{Proposition}[cf. {\cite[Prop.~4.1(1)]{FutOvs2010}}]
\label{prp:Galois-Ring-Equivalent-Condition}
Let $\SX\subset\SK$ and let $\SU$ be the subring of $\SK$ generated by $\Ga\cup \SX$. Then $\SU$ is a Galois $\Ga$-ring in $\SK$ iff $\cup_{X\in \SX}\Supp_\SM(X)$ generates $\SM$ as a monoid.
\end{Proposition}

\begin{proof}
The ``only if'' part is obvious since $\Supp_\SM(ab)\subset\Supp_\SM(a)\Supp_\SM(b)$ for any $a,b\in\SL$.
For the converse, by Lemma \ref{lem:Galois-Ring-Equivalent-Condition}(iv)--(v) it suffices to show that if $x\in X$ and $\mu\in\Supp_\SM(x)$ then $[\mu]\in KxK$. Write $x=\sum_i [a_i\mu_i]$ for some $\mu_i\in\SM$ from pairwise distinct $G$-orbits in $\SM$, where $a_i\in L^{G_{\mu_i}}$.
Then $KxK\subset \sum_i K[a_i\mu_i]K = \bigoplus_i K[\mu_i]K$ by Lemma \ref{lem:Galois-Ring-Equivalent-Condition}(ii)--(iii). Thus $(KxK)\cap K[\mu_i]K$ is either zero or $K[\mu_i]K$. But the former would imply $KxK\subset \bigoplus_{j\neq i} K[\mu_j]K$ contradicting $\mu_i\in\Supp_\SM(x)$. So $[\mu_i]\in KxK$ for all $i$.
\end{proof}

\begin{Lemma} \label{lem:galois-ring-trivial}
The following statements hold.
\begin{enumerate}[{\rm (i)}]
\item $\SK$ is a Galois $\Ga$-ring in $\SK$.
\item Let $\SU_1\subset\SU_2$ be subrings of $\SK$. If $\SU_1$ is a Galois $\Ga$-ring in $\SK$ then so too is $\SU_2$.
\item $(\La\ast\SM)^G$ is Galois $\Ga$-ring in $\SK$.
\end{enumerate}
\end{Lemma}

\begin{proof}
(i)(ii): By Definition \ref{def:Galois-ring}.

\smallskip\noindent
(iii): By Lemma \ref{lem:Galois-Ring-Equivalent-Condition}(iv)(v), since $[\mu]\in(\La\ast\SM)^G$ for all $\mu\in\SM$.
\end{proof}

\begin{Example} \label{ex:symmetric-difference-operators-2}
Continuing Example \ref{ex:symmetric-difference-operators-1}, fix $f=(f_i)_{i=1}^n\in\La^n$ with $\si(f_i)=f_{\si(i)}$ for all $i$, and define $X_f\in L\ast \SM$ by
\begin{equation}
X_f = \sum_{i=1}^n \frac{f_i}{\displaystyle \prod_{1\le j\le n,\, j\neq i} (x_j-x_i)} \mu_i.
\end{equation}
It is easy to see that $\si(X_f)=X_f$ for all $\si\in S_n$, hence $X_f\in \SK$.
Let $U(f)$ be the subring of $\SK$ generated by $\Ga\cup\{X_f\}$. Since $\Supp_\SM(X_f)=\{\mu_1,\mu_2,\ldots,\mu_n\}$ which generates the monoid $\SM\cong\N^n$, Proposition \ref{prp:Galois-Ring-Equivalent-Condition} implies that $U(f)$ is a Galois $\Ga$-ring in $\SK$.
\end{Example}

\subsection{Galois orders}
In this subsection we define the notions of (co-)standard and (co-)principal Galois orders, and prove sufficient conditions for a Galois ring to be a Galois order.

\begin{Definition}[\cite{FutOvs2010}] \label{def:Galois-order}
A Galois $\Ga$-ring $\SU$ in $\SK$ is a \emph{left} (respectively \emph{right}) \emph{Galois $\Ga$-order in $\SK$} if for any finite-dimensional left (respectively right) $K$-subspace $W\subset\SK$, $W\cap \SU$ is a finitely generated left (respectively right) $\Ga$-module.
A Galois $\Ga$-ring $\SU$ in $\SK$ is a \emph{Galois $\Ga$-order in $\SK$} if $\SU$ is a left and right Galois $\Ga$-order in $\SK$.
\end{Definition}

\begin{Definition}
A commutative subring $C$ of a ring $\SA$ is \emph{maximal commutative in $\SA$} if $C$ is not properly contained in any commutative subring of $\SA$.
\end{Definition}

The following result is expected from \cite[Thm.~5.2(2)]{FutOvs2010} since $\Ga$ is integrally closed.

\begin{Proposition} \label{prp:GO-Gamma-maximal-commutative}
$\Ga$ is maximal commutative in any left or right Galois $\Ga$-order $\SU$ in $\SK$.
\end{Proposition}

\begin{proof}
Let $x\in\SU$, $[x,\Ga]=0$. Write $x=a_1\mu_1+a_2\mu_2+\cdots+a_n\mu_n$ for some $a_i\in L$ and $\mu_i\in\SM$. Then for $\ga\in\Ga$, $0=[x,\ga]=\sum_i a_i(\ga-\mu_i(\ga))\mu_i$. Hence $\mu_i(\ga)=\ga$ for all $\ga\in \Ga$. By \eqref{it:G1}, $n=1$ and $\mu_1=1_{\Aut(\La)}$. This shows that $x\in K$. Since $\SU$ is a left or right Galois $\Ga$-order in $\SK$, $K\cap\SU$ is a finitely generated $\Ga$-module. Since $\Ga$ is integrally closed, $K\cap\SU=\Ga$. Thus $x\in \Ga$.
\end{proof}

\begin{Corollary}
$\SK$ is a Galois $\Ga$-order in $\SK$ iff $\La$ is a field.
\end{Corollary}

\begin{proof}
($\Rightarrow$): Proposition \ref{prp:GO-Gamma-maximal-commutative} implies that $\Ga=K$, hence $\La=L$ since $\La$ is integral over $\Ga$.

\smallskip\noindent
($\Leftarrow$): If $\La=L$ then $\Ga=K$, hence this direction is immediate by Definition \ref{def:Galois-order}.
\end{proof}

\begin{Lemma} \label{lem:sub-order}
Let $\SU_1$ and $\SU_2$ be two Galois $\Ga$-rings in $\SK$ such that $\SU_1\subset \SU_2$. If $\SU_2$ is a Galois $\Ga$-order in $\SK$, then so too is $\SU_1$.
\end{Lemma}

\begin{proof}
Immediate by the definition of Galois order and that $\Ga$ is a noetherian ring.
\end{proof}

Part (ii) of the next lemma is probably well-known but we could not find a proof in the literature.

\begin{Lemma} \label{lem:Dedekind}
Let $A$ be an integral domain, $F$ be a field, and $\si_1,\si_2,\ldots,\si_n$ be pairwise distinct injective ring homomorphisms from $A$ to $F$. Then:
\begin{enumerate}[{\rm (i)}]
\item $\{\si_1,\si_2,\ldots,\si_n\}$ is a linearly independent subset of the $F$-vector space $F^A$;
\item There exists $(a_1,a_2,\ldots,a_n)\in A^n$ such that the determinant of $\big(\si_j(a_i)\big)_{i,j=1}^n$ is nonzero.
\end{enumerate}
\end{Lemma}

\begin{proof}
(i): This follows from Dedekind's Independence Theorem, see e.g. \cite[Sec.~4.14]{Jacobson1}.

\smallskip\noindent
(ii): We use induction on $n$. For $n=1$, we may take $a_1=1$. For $n>1$, assume that we found $(a_1,a_2,\ldots,a_{n-1})\in A^{n-1}$ such that $\big(\si_j(a_i)\big)_{i,j=1}^{n-1}$ has nonzero determinant. Then the system
\begin{equation}\label{eq:ded1}
\si_1(a_i) x_1 + \si_2(a_i) x_2 + \cdots + \si_{n-1}(a_i) x_{n-1} = \si_n(a_i),\qquad i=1,2,\ldots,n-1,
\end{equation}
has a solution $(x_1,x_2,\ldots,x_{n-1})\in F^{n-1}$.
By part (i), there exists $a_n\in A$ such that
\begin{equation}\label{eq:ded2}
\big(\si_n - \sum_{i=1}^{n-1}x_i\si_i\big)(a_n)\neq 0.
\end{equation}
Performing column operations and using \eqref{eq:ded1}, we can eliminate the rightmost column except for the bottom right entry:
\[
\begin{vmatrix}
\si_1(a_1) & \si_2(a_1) & \cdots & \si_n(a_1)\\
\si_1(a_2) & \si_2(a_2) & \cdots & \si_n(a_2)\\
  \vdots   & \vdots     & \ddots & \vdots \\
\si_1(a_n) & \si_2(a_n) & \cdots & \si_n(a_n)\\
\end{vmatrix}
=
\begin{vmatrix}
\si_1(a_1) & \si_2(a_1) & \cdots & \si_{n-1}(a_1) & 0\\
\si_1(a_2) & \si_2(a_2) & \cdots & \si_{n-1}(a_2) & 0\\
  \vdots   & \vdots     & \vdots & \vdots         & \vdots\\
\si_1(a_n) & \si_2(a_n) & \cdots & \si_{n-1}(a_n) & \si_n(a_n) - \sum_{i=1}^{n-1} x_i \si_i(a_n)\\
\end{vmatrix}
\]
This determinant is nonzero by \eqref{eq:ded2} and the induction hypothesis.
\end{proof}

\begin{Definition} \label{def:evaluation}
For $X=\sum_{\mu\in\SM} x_\mu \mu \in \SL$ and $a\in L$ we define the \emph{evaluation of $X$ at $a$} to be
\begin{equation} \label{eq:evaluation}
X(a)= \sum_{\mu\in\SM} x_\mu \cdot \big(\mu(a)\big) \in L.
\end{equation}
\end{Definition}

\begin{Lemma} \label{lem:evaluation}
Evaluation of $\SL$ on $L$ satisfies the following properties.
\begin{enumerate}[{\rm (i)}]
\item $(X,a)\mapsto X(a)$ is a $\Z$-bilinear map $\SL\times L\to L$.
\item $X\big(Y(a)\big)=(XY)(a)$ for all $X,Y\in\SL$ and $a\in L$.
\item If $X\in\SK$ and $a\in K$ then $X(a)\in K$.
\end{enumerate}
\end{Lemma}
\begin{proof}
(i)(ii): Easy to check.

\smallskip\noindent
(iii): Let $X\in\SK$. Since $X\in \SL$ we can write $X=\sum_{\mu\in\SM}x_\mu \mu$
for some unique $x_\mu\in L$, at most finitely many nonzero. Since $X\in\SK=\SL^G$ we have $g(x_\mu) = x_{{}^g\mu}$ for all $g\in G$ and all $\mu\in\SM$. Thus for any $a\in K$ and $g\in G$:
\[g \big(X(a)\big) =\sum_{\mu\in\SM} g (x_\mu) \cdot g \big(\mu(a)\big) = \sum_{\mu\in\SM} x_{{}^g \mu} \cdot {}^g {\mu} \big(g(a)\big) = \sum_{\mu\in\SM} x_\mu \cdot \mu(a) = X(a).\]
This shows that $X(a)\in L^G=K$.
\end{proof}

\begin{Example} \label{ex:symmetric-difference-operators-3}
Continuing Example \ref{ex:symmetric-difference-operators-2}, we have for any $g\in L$,
\[X_f(g)=\sum_{i=1}^n \frac{f_i\mu_i(g)}{\displaystyle \prod_{1\le j\le n,\, j\neq i} (x_j-x_i)}.\]
\end{Example}

We are ready to prove the main result of this section.

\begin{Theorem} \label{thm:KG-Galois-order}
\begin{equation} \label{eq:KG-definition}
\SK_\Ga \eqdef \big\{X\in \SK\mid X(\ga)\in \Ga \;\;\forall \ga\in \Ga\big\}
\end{equation}
is a Galois $\Ga$-order in $\SK$.
\end{Theorem}

\begin{proof}
If $X,Y\in\SK_\Ga$ then $(XY)(\ga)=X\big(Y(\ga)\big)\in\Ga$ and $(X\pm Y)(\ga)=X(\ga)\pm Y(\ga)\in\Ga$ for all $\ga\in \Ga$. Thus $\SK_\Ga$ is a subring of $\SK$. Since $\mu(\Ga)\subset\La$ for any $\mu\in\SM$ we have $X(\Ga)\subset\La$ for any $X\in\La\ast\SM$. By Lemma \ref{lem:evaluation}(iii), $X(\Ga)\subset\La\cap K=\Ga$ for any $X\in(\La\ast\SM)^G$. Thus $(\La\ast\SM)^G\subset\SK_\Ga$. By Lemma \ref{lem:galois-ring-trivial}(iii) and (ii), $\SK_\Ga$ is a Galois $\Ga$-ring in $\SK$.

To show that $\SK_\Ga$ is a Galois $\Ga$-order, let $W$ be a finite-dimensional left $K$-subspace of $\SK$. Then $LW$ is a finite-dimensional left $L$-subspace of $\SL$.
Suppose that $\SK_\Ga\cap LW$ is finitely generated as a left $\Ga$-module. Then, since $\Ga$ is noetherian, $\SK_\Ga\cap W$ is finitely generated as a left $\Ga$-module. Similarly on the right. So it suffices to show that for any finite-dimensional left (respectively right) $L$-subspace $W\subset \SL$, the set $\SK_\Ga\cap W$ is finitely generated as a left (respectively right) $\Ga$-module.
Since $L\mu=\mu L$ for any $\mu\in\SM$, left and right $L$-subspaces of $\SL$ are the same thing.

Let $\{w_1,w_2,\ldots,w_d\}$ be a left $L$-basis for $W$. Write each $w_i$ as a finite linear combination of elements of $\SM$. Thus $W$ is contained in a subspace of the form $\tilde{W}=\bigoplus_{i=1}^n L\mu_i$ where $\mu_i\in \SM$. Assume for a moment that $\SK_\Ga\cap\tilde{W}$ is finitely generated as a left and right $\Ga$-module. Then, since $\Ga$ is noetherian, the submodule $\SK_\Ga\cap W$ is also finitely generated. This shows that we may without loss of generality assume that $W=\bigoplus_{i=1}^n L\mu_i$ for some finite subset $\{\mu_1,\mu_2,\ldots,\mu_n\}$ of $\SM$.
Define a map
\[\psi:\SK_\Ga\cap W\to L^n\]
by
\[\psi\big(\sum_{i=1}^n a_i\mu_i\big) = (a_1,a_2,\ldots,a_n)\]
Clearly $\psi$ is injective. If we view $L^n$ as a $\Ga$-bimodule via
\begin{equation} \label{eq:proof-bimodule}
\ga(a_1,a_2,\ldots,a_n)\ga'=\big(\ga a_1\mu_1(\ga'),\ga a_2\mu_2(\ga'),\ldots,\ga a_n\mu_n(\ga')\big)
\end{equation}
for $\ga,\ga'\in\Ga$, then $\psi$ is a $\Ga$-bimodule homomorphism. Let  $a=(a_1,a_2,\ldots,a_n)\in L^n$ belong to the image of $\psi$. Then $\sum_i a_i\mu_i(\ga)\in \Ga$ for all $\ga\in \Ga$. By Lemma \ref{lem:Dedekind}(ii) applied to $\mu_i\big|_\Ga$, which are pairwise distinct by Assumption \eqref{it:G1}, there are $\ga_i\in\Ga$ such that
$d\eqdef \det A\neq 0$, where $A=\big(\mu_j(\ga_i)\big)_{i,j=1}^n$. Now
$\sum_{i=1}^n a_i\mu_i(\ga_j) = \tilde{\ga}_j$
for some $\tilde{\ga}=(\tilde{\ga}_1,\tilde{\ga}_2,\ldots,\tilde{\ga}_n)\in\Ga^n$. That is, $A\cdot a = \tilde{\ga}$.
All the entries of $A$ are in $\La$, hence that is true for the adjugate matrix $A^\prime$ satisfying $A^\prime\cdot A = d\cdot I_n$. So $d\cdot a=A^\prime\cdot \tilde{\ga}\in \La^n$ for $i=1,2,\ldots,n$. That is, $a\in \frac{1}{d}\La^n$.
Thus the image of $\psi$ is contained in $\frac{1}{d}\La^n$ which is a $\Ga$-subbimodule of $L^n$ which is finitely generated as a left $\Ga$-module and, with respect to the twisted action \eqref{eq:proof-bimodule}, as a right $\Ga$-module. Since $\Ga$ is noetherian, $\SK_\Ga\cap W$ is also a finitely generated left and right $\Ga$-module.
\end{proof}

\begin{Definition}
 $\SK_\Ga$ is the \emph{standard Galois $\Ga$-order in $\SK$}.
\end{Definition}

\begin{Corollary} \label{cor:GO-criterion}
If $\SU$ is a Galois $\Ga$-ring in $\SK$ and $\SU\subset\SK_\Ga$, then $\SU$ is a Galois $\Ga$-order in $\SK$.
\end{Corollary}

\begin{proof}
By Theorem \ref{thm:KG-Galois-order} and Lemma \ref{lem:sub-order}.
\end{proof}

\begin{Definition}
If $\SU$ is a Galois $\Ga$-order in $\SK$ such that $\SU\subset\SK_\Ga$, then $\SU$ is a \emph{principal Galois $\Ga$-order in $\SK$}.
\end{Definition}

\begin{Corollary} \label{cor:GO-criterion-gen}
If $\SU$ is a Galois $\Ga$-ring in $\SK$ and $\SU$ is generated as a ring by a subset $\SX$ such that $X(\ga)\in\Ga$ for all $(\ga,X)\in\Ga\times\SX$, then $\SU$ is a principal Galois $\Ga$-order in $\SK$.
\end{Corollary}

\begin{proof}
By \eqref{eq:KG-definition}, $\SX\subset\SK_\Ga$. Since $\SK_\Ga$ is a ring, $\SU\subset\SK_\Ga$. Now use Corollary \ref{cor:GO-criterion}.
\end{proof}

\begin{Definition}
For a linear character $\chi:G\to K^\times$, the set of \emph{$\chi$-relative $G$-invariants} in $\La$ is
\begin{equation}
\La^G_\chi=\{a\in \La\mid g(a)=\chi(g)a\;\forall g\in G\}.
\end{equation}
\end{Definition}

The following is a useful method to produce interesting elements of $\SK_\Ga$.

\begin{Lemma} \label{lem:principal-test}
If $\chi:G\to K^\times$ is a linear character such that $\La^G_\chi=\Ga d_\chi$ for some $d_\chi\in \La$, then
\begin{equation}
\SK\cap \frac{1}{d_\chi}(\La\ast\SM)\subset \SK_\Ga.
\end{equation}
\end{Lemma}

\begin{proof}
If $x\in\SK\cap\frac{1}{d_\chi}(\La\ast\SM)$ then for any $\ga\in \Ga$ we have $ x(\ga)\in \big(\frac{1}{d_\chi}\La\big)\cap K = \frac{1}{d_\chi}\La^G_\chi=\Ga$.
\end{proof}

\begin{Example} \label{ex:symmetric-difference-operators-4}
Continuing Example \ref{ex:symmetric-difference-operators-2},
let $\sgn:S_n\to\{1,-1\}$ be the sign character and $d_{\sgn}=\prod_{1\le i<j\le n} (x_i-x_j)\in\La$ be the Vandermonde determinant. Since any alternating polynomial is divisible by $d_{\sgn}$ we have $\La^{S_n}_{\sgn}=\Ga d_{\sgn}$. Observe that $d_\chi X_f \in \La\ast\SM$. Hence $X_f\in \SK\cap\frac{1}{d_\chi}(\La\ast\SM)$. By Lemma \ref{lem:principal-test} and Corollary \ref{cor:GO-criterion-gen}, $U(f)$ is a principal Galois $\Ga$-order in $\SK$.
\end{Example}

We can now prove parts (i) and (ii) of Theorem \ref{thm:I}.

\begin{proof}[Proof of Theorem \ref{thm:I}(i)(ii)]
By Proposition \ref{prp:Galois-Ring-Equivalent-Condition}, Corollary \ref{cor:GO-criterion-gen} and Proposition \ref{prp:GO-Gamma-maximal-commutative}.
\end{proof}

\begin{Remark} \label{rem:dagger}
There is an anti-isomorphism $\dagger:L\ast\SM\to L\ast(\SM^{-1})$, $f^\dagger=f$ for $f\in L$ and $\mu^\dagger=\mu^{-1}$ for $\mu\in\SM$. Then $\dagger$ commutes with the action of $G$ on $\SL$, hence $\SK^\dagger=\big(L\ast(\SM^{-1})\big)^G$. Moreover, $\SU$ is a Galois $\Ga$-ring (respectively $\Ga$-order) in $\SK$ iff $\SU^\dagger$ is a Galois $\Ga$-ring (respectively $\Ga$-order) in $\SK^\dagger$.
However it is not true in general that $(\SK_\Ga)^\dagger=(\SK^\dagger)_\Ga$. Thus this gives an opposite way to produce Galois $\Ga$-orders in $\SK$ as images under $\dagger$ of principal Galois $\Ga$-orders in $\SK^\dagger$.
\end{Remark}

\begin{Definition} \label{def:co-principal}
We call ${}_\Ga\SK\eqdef\big((\SK^\dagger)_\Ga\big)^\dagger\subset \SK$ the \emph{co-standard Galois $\Ga$-order in $\SK$}. 
Equivalently,
\begin{equation}
{}_\Ga\SK=\{X\in\SK\mid X^\dagger(\ga)\in\Ga\;\forall\ga\in\Ga\}.
\end{equation}
If $\SU$ is a Galois $\Ga$-order in $\SK$ contained in ${}_\Ga\SK$ then $\SU$ is a \emph{co-principal Galois $\Ga$-order in $\SK$}.
\end{Definition}

We have the following opposite versions of Corollary \ref{cor:GO-criterion-gen} and Lemma \ref{lem:principal-test}.

\begin{Corollary} \label{cor:anti-criterion}
If $\SU$ is a Galois $\Ga$-ring in $\SK$ and $\SU$ is generated as a ring by a subset $\SX$ such that $X^\dagger(\ga)\in\Ga$ for all $\ga\in\Ga$ and all $X\in\SX$, then $\SU$ is a co-principal Galois $\Ga$-order in $\SK$.
\end{Corollary}

\begin{Lemma} \label{lem:co-principal-test}
If $\chi:G\to K^\times$ is a linear character such that $\La^G_\chi=\Ga d_\chi$ for some $d_\chi\in \La^G_\chi$ then 
\begin{equation}
(\La\ast\SM)\frac{1}{d_\chi}\cap \SK\subset (\SK_\Ga)^\dagger.
\end{equation}
\end{Lemma}

\section{Gelfand-Zeitlin modules} \label{sec:GZ-modules}

In \cite{EarMazVis2017} the authors constructed canonical simple Gelfand-Zeitlin modules was defined in the setting of orthogonal Gelfand-Zeitlin algebras. In this section we generalize this to an arbitrary principal Galois $\Ga$-order $\SU$ in $\SK$. In Section \ref{sec:OGZ} we show that orthogonal Gelfand-Zeitlin algebras are examples of principal Galois $\Ga$-orders.

In addition to Assumptions \eqref{it:G1}--\eqref{it:G3}, in this section we assume that:
\begin{align}
\tag{A4} \label{it:G4} &\text{$\La$ is finitely generated over an algebraically closed field $\K$ of characteristic zero,} \\ 
\tag{A5} \label{it:G5} &\text{$G$ and $\SM$ act by $\K$-algebra automorphisms on $\La$.}
\end{align}
Let $\hat\Ga$ be the set of all $\Ga$-characters, i.e. $\K$-algebra homomorphisms $\xi:\Ga\to\K$.

\begin{Definition}
Let $\SU$ be a Galois $\Ga$-ring in $\SK$. A left $\SU$-module $M$ is said to be a \emph{Gelfand-Zeitlin module (with respect to $\Ga$)} if $\Ga$ acts locally finitely on $M$. Equivalently,
\[M=\bigoplus_{\xi\in\hat{\Ga}} M_\xi, \qquad M_\xi=\{v\in M\mid (\ker \xi)^Nv=0, N\gg 0\}.\]
Symmetrically one defines the notion of a right Gelfand-Zeitlin module.
\end{Definition}

The following is standard, see e.g. \cite{DroFutOvs1994}.
We provide some details for the convenience of the reader.

\begin{Lemma} \label{lem:HC-GZ}
Let $\SU$ be a Galois $\Ga$-ring in $\SK$.
\begin{enumerate}[{\rm (i)}]
\item Any submodule and any quotient of a Gelfand-Zeitlin module is a Gelfand-Tsetlin module.
\item Any $\SU$-module generated by generalized weight vectors is a Gelfand-Zeitlin module.
\end{enumerate}
\end{Lemma}
\begin{proof}
(i): Obvious for submodules. If $\Ga$ acts locally finite on $v\in M$ then it does so on $v+N\in M/N$ for any submodule $N$.

\smallskip\noindent
(ii): It suffices to prove this for a left cyclic $\SU$-module, say $M=\SU v$. Let $u\in \SU$. Since $\Ga$ is a Harish-Chandra subalgebra of $\SU$, $\Ga u\Ga$ is finitely generated as a right $\Ga$-module. So $\Ga u \Ga = \sum_{i=1}^n \ga_i u \Ga$ for some $\ga_i\in \Ga$. Hence $\Ga u v \subset\Ga u \Ga v \subset \sum_{i=1}^n \ga_i u \Ga v$ which is finite-dimensional over $\K$. The proof for right modules is symmetric.
\end{proof}

Let $\Ga^\ast$ be the set of all $\K$-linear maps from $\Ga$ to $\K$. Note that $\Ga^\ast$ is a right $\SK_\Ga$-module with respect to the action $\tau X\eqdef\tau\circ X$ for $\tau\in \Ga^\ast$ and $X\in \SK_\Ga$, where $X$ is viewed as a function from $\Ga$ to $\Ga$ via evaluation. Similarly, $\Ga^\ast$ is a left ${}_\Ga\SK$-module via $X\tau=\tau\circ(X^\dagger)$. Thus by restriction, $\Ga^\ast$ is a right (left) $\SU$-module for any (co-)principal Galois $\Ga$-order $\SU$ in $\SK$.
We now prove the main theorem of this section, which in particular proves Theorem \ref{thm:I}(iii).

\begin{Theorem} \label{thm:principal-nonempty-fiber}
Let $\xi\in\hat{\Ga}$ be any character. 
\begin{enumerate}[{\rm (i)}]
\item If $\SU$ is a principal Galois $\Ga$-order in $\SK$, then the right cyclic $\SU$-module $\xi\SU$ has a unique simple quotient $V(\xi)$. Moreover, $V(\xi)$ is a Gelfand-Zeitlin module over $\SU$ with $V(\xi)_\xi\neq 0$.
\item If $\SU$ is a co-principal Galois $\Ga$-order in $\SK$, then the left cyclic $\SU$-module $\SU\xi$ has a unique simple quotient $V'(\xi)$. Moreover, $V'(\xi)$ is a Gelfand-Zeitlin module over $\SU$ with $V'(\xi)_\xi\neq 0$.
\end{enumerate}
\end{Theorem}

\begin{proof}
(i): Put $M(\xi)=\xi\SU$. Since $\xi$ is a weight vector of weight $\xi$, $M(\xi)$ is a right Gelfand-Zeitlin $\SU$-module by Lemma \ref{lem:HC-GZ}.
Let $\Fm=\ker\xi$. As in \cite[Prop.~5]{EarMazVis2017} we have
\[\Hom_{\Ga}(\Ga/\Fm,\Ga^\ast) \cong \Hom_\K(\Ga/\Fm\otimes_{\Ga}\Ga,\K)\cong \K\]
since $\Ga/\Fm\cong\K$. Thus $\xi$ is the unique (up to scalar) weight vector in $\Ga^\ast$ of weight $\xi$. Thus every proper submodule has zero intersection with the generalized weight space $M(\xi)_\xi$. Therefore, by Lemma \ref{lem:HC-GZ}, the sum $N(\xi)$ of all proper submodules is itself a proper submodule and $V(\xi)=M(\xi)/N(\xi)$ is a Gelfand-Zeitlin module  with $V(\xi)_\xi\neq 0$ since $\xi\notin N(\xi)$.

\smallskip\noindent
(ii): The proof is symmetric.
\end{proof}

\begin{Definition}
Let $\xi\in\hat{\Ga}$, $\SU$ (respectively $\SU'$) be a principal (respectively co-principal) Galois $\Ga$-order in $\SK$. Then
\begin{enumerate}[{\rm (i)}]
\item $V(\xi)$ is the \emph{canonical simple right Gelfand-Zeitlin $\SU$-module} associated to $\xi$,
\item $V'(\xi)$ is the \emph{canonical simple left Gelfand-Zeitlin $\SU'$-module} associated to $\xi$. 
\end{enumerate}
\end{Definition}

\section{Rational Galois orders} \label{sec:rat-GO}

In this section we introduce a general construction of a large class of principal Galois orders we call \emph{rational Galois orders}. This class naturally includes $U(\Fgl_n)$, level $p$ Yangians $Y_p(\Fgl_n)$, type $A$ finite W-algebras $W(\pi)$, OGZ algebras $U(\boldsymbol{r})$ as well as their respective parabolic subalgebras.

\subsection{Construction from finite reflection groups}

Consider a finite reflection group $G\subset\GL(V)$, where $V$ is a vector space over a field $\K$ of characteristic not dividing $|G|$. By a \emph{reflection group} \cite{Broue,Terao1989} we mean a group generated by pseudo-reflections, i.e. linear transformations $g$ with $\op{codim}\ker(1-g)=1$. 
Let $\La=S=S(V^\ast)$ be the algebra of polynomial functions on $V$, and
 $\Ga=R=S^G$ be the subring of $G$-invariants. Let $L=\Frac S$ and $K=\Frac R$ be the respective fields of fractions. Let $V$ act on $S$ by translations, $t_w(p)(v)=p(v-w)$ for $v,w\in V,\, p\in S$. Let $S\ast V$ be the skew group algebra defined as the free left $S$-module on the set $\{t_v\}_{v\in V}$ with multiplication $pt_v\cdot qt_w = (pt_v(q))t_{v+w}$. Since $g\circ t_v\circ g^{-1}= t_{g(v)}$, the group $G$ acts on $S\ast V$ by $\K$-algebra automorphisms. Let $\hat{G}$ be the group of linear characters $\chi:G\to\K^\times$. 
For $\chi\in\hat{G}$, let $S^G_\chi=\{f\in S\mid g(f)=\chi(g)f\;\forall g\in G\}$ be the subring of \emph{relative invariants} and put
\begin{equation}
d_\chi = \prod_{H\in \SA(G)} (\al_H)^{a_H}
\end{equation}
where $\SA(G)=\{\ker(1-g)\mid g\in G\}$ is the set of reflecting hyperplanes in $V$, $\al_H\in V^\ast$ is a choice of linear form with $\ker \al_H=H$, and $a_H$ is the least non-negative integer with $\chi(s_H)=\big(\det(s_H^\ast)\big)^{a_H}$, where $s_H$ is a fixed generator of the stabilizer of $H$ in $G$. The key result we need is the following theorem from \cite{Terao1989}.

\begin{Theorem}[{\cite[Thm.~2.5]{Terao1989}}] \label{thm:Terao}
For any $\chi\in\hat{G}$ we have $S^G_\chi = R d_\chi$.
\end{Theorem}

This allows us to construct principal Galois orders as follows.

\begin{Theorem} \label{thm:rational-Galois-order}
Let $G\subset \GL(V)$ be a finite reflection group and $\SX$ be a subset of $L\ast V$ such that:
\begin{enumerate}[{\rm (i)}]
\item \label{it:RGOi} $g(X)=X$ for all $g\in G$ and all $X\in \SX$,
\item \label{it:RGOii} for all $X\in \SX$ there exists $\chi\in\hat{G}$ such that $d_\chi\cdot X\in S\ast V$ (respectively $X\cdot d_\chi\in S\ast V$).
\end{enumerate}
Then the subring $\SU(G,\SX)$ of $L\ast V$ generated by $R\cup \SX$ is a principal (respectively co-principal) Galois $R$-order in $(L\ast\SM)^G$ where $\SM\subset V$ is the submonoid generated by $\cup_{X\in \SX}\Supp_{\SM}(X)$.
\end{Theorem}

\begin{proof}
Properties \eqref{it:G1}--\eqref{it:G3} are straightforward to verify. Put $\SK=(L\ast\SM)^G$. By Theorem \ref{prp:Galois-Ring-Equivalent-Condition}, $\SU(G,\SX)$ is a Galois $R$-ring in $\SK$. That $\SU(G,\SX)\subset \SK_R$ follows from Theorem \ref{thm:Terao} and Lemma \ref{lem:principal-test}.
By Corollary \ref{cor:GO-criterion-gen}, $\SU(G,\SX)$ is a principal Galois $R$-order in $\SK$.
The opposite case follows from Corollary \ref{cor:anti-criterion}, Lemma \ref{lem:co-principal-test}.
\end{proof}

\begin{Definition} \label{def:rational-GO}
Call $\SU(G,\SX)$ the \emph{rational (respectively co-rational) Galois order} associated to $G$ and $\SX$.
\end{Definition}

\begin{Example} \label{ex:symmetric-difference-operators-5}
The algebra $U(f)$ defined in Examples
\ref{ex:symmetric-difference-operators-1},
\ref{ex:symmetric-difference-operators-2},
\ref{ex:symmetric-difference-operators-3},
\ref{ex:symmetric-difference-operators-4},
can be viewed as the rational Galois order $\SU\big(S_n,\{X_f\}\big)$ with $V=\bigoplus_{i=1}^n \C\mu_i$ and coordinate functions $x_i(\mu_j)=\delta_{ij}$.
\end{Example}

\begin{Remark}
Any co-rational Galois order is isomorphic to a rational Galois order via the the restriction of the automorphism $L\ast V\to L\ast V$, $t_v\mapsto t_{-v}$, $f\mapsto f^-$ where $f^-(v)=f(-v)$.
\end{Remark}

\subsection{Finite W-algebras of type $A$} \label{sec:finite-W}
Let $\K$ be an algebraically closed field of characteristic zero.
Let $n$ be a positive integer, $\pi=(p_1,p_2,\ldots,p_n)\in\Z^n$ where $1\le p_1\le p_2\le \cdots\le p_n$, and let $W(\pi)$ be the corresponding finite W-algebra of type $A$, see \cite{FutMolOvs2010}. There is a generating set for $W(\pi)$ consisting of the coefficients of certain polynomials $A_i(u),B^\pm_k(u)\in W(\pi)[u]$, $i=1,2,\ldots,n;\, k=1,2,\ldots,n-1$. We recall the realization of $W(\pi)$ as a Galois order from \cite{FutMolOvs2010} and show that it is a special case of a rational Galois order, as defined above. 

Let $\La=\K[x_{ri}^k\mid 1\le i\le r\le n;\, 1\le k\le p_i]$ be a polynomial algebra in $N=np_1+(n-1)p_2+\cdots+p_n$ variables, $G=S_{p_1}\times S_{p_1+p_2}\times\cdots\times S_{p_1+p_2+\cdots+p_n}$ acting on $\La$ by letting the $r$:th component permute all the $p_1+p_2+\cdots+p_r$ variables $\{x_{ri}^k\mid 1\le i\le k;\, 1\le k\le p_i\}$. Let $\SM\cong\Z^{N-(p_1+p_2+\cdots+p_n)}$ be the free abelian group, written multiplicatively, with basis $\{\delta_{ri}^k\mid 1\le i\le r\le n-1;\, 1\le k\le p_i\}$ acting faithfully on $\La$ by 
\[
\delta_{ri}^k(x_{sj}^l)=
\begin{cases}
x_{sj}^l - 1 & \text{if $(r,i,k)=(s,j,l)$,}\\
x_{sj}^l & \text{otherwise.}
\end{cases}
\]
Properties \eqref{it:G1}--\eqref{it:G5} are easily verified, and we form the corresponding algebra $\SK=(L\ast \SM)^G$ where $L=\Frac\La$.
By \cite[Lem.~3.5]{FutMolOvs2010}, there exists an injective $\K$-algebra homomorphism
\begin{equation}
 i:W(\pi)\to \SK
\end{equation}
given by the following equalities in $\SK[u]$:
\begin{align}
i\big(A_j(u)\big) &= A_j(u)1 \\
i\big(B^\pm_r(u)\big) &= \sum_{(l,j)} 
(\delta_{rj}^l)^{\pm 1} \cdot 
X_{rlj}^\pm(u)
\end{align}
where
\begin{equation}
X_{rlj}^\pm(u) = \mp \frac{\prod_{(k,i)\neq (l,j)} (u+x_{ri}^k)\prod_{m,q}(x_{r\pm 1,q}^m - x_{rj}^l)}{\prod_{(k,i)\neq (l,j)} (x_{ri}^k-x_{rj}^l)}.
\end{equation}
Moreover $i$ maps the Gelfand-Zeitlin subalgebra of $W(\pi)$ to $\Ga=\La^G$.

\begin{proof}[Proof of Theorem \ref{thm:intro-finite-W}]
(i): Let $V=\K^N$ and identify $\La$ with $S(V^{\ast})$ in the obvious way. $G$ can be identified with the reflection group in $\GL(V)$ permuting coordinates.  
The linear forms $x_{ri}^k-x_{rj}^s$ vanish on the reflecting hyperplane corresponding to the transposition in $S_{p_1+p_2+\cdots+p_r}$ interchanging those variables.
 Thus the Jacobian $d_{\sgn}$ corresponding to the sign ($=$determinant) character $\sgn:G\to\{\pm 1\}$ is the product of the $n$ Vandermonde determinants $\CV_r$ in the variables $x_{ri}^k$. 
The key observation is that $X_{rlj}^\pm(u)\cdot d_{\sgn}$ is a polynomial. Let $\SX$ be the set of coefficients of powers of $u$ in $i\big(B_k^+(u)\big)$ and $i\big(B_j^-(u)\big)$ for $k=1,2,\ldots,n-1$ and $j\in J$. Then Theorem \ref{thm:rational-Galois-order} implies that $W_J(\pi)$ is isomorphic to the rational Galois order $\SU(G,\SX)$.

\smallskip\noindent
(ii): Follows from part (i), Theorem \ref{thm:rational-Galois-order} and Theorem \ref{thm:principal-nonempty-fiber}.
\end{proof}

\subsection{Level $p$ Yangians $Y_p(\Fgl_n)$ and enveloping algebra $U(\Fgl_n)$}
These are special cases of finite W-algebras, $Y_p(\Fgl_n)\cong W(\pi)$ where $\pi=(p,p,\ldots,p)$, see \cite{BruKle2006,FutMolOvs2010,RagSor1999}. Specializing further, the enveloping algebra $U(\Fgl_n)$ is isomorphic to $Y_1(\Fgl_n)$. Hence $Y_p(\Fgl_n)$ and $U(\Fgl_n)$, along with their respective parabolic subalgebras, are also examples of rational Galois orders.

\subsection{Orthogonal Gelfand-Zeitlin algebras of type $A$}
\label{sec:OGZ}
Let $\K$ be algebraically closed of characteristic zero.
Orthogonal Gelfand-Zeitlin (OGZ) algebras $U(\boldsymbol{r})$ of type $A$ were defined by Mazorchuk in \cite{Mazorchuk1999}. Quantum analogs were defined in \cite{Hartwig2017}. They can be understood directly as examples of Galois rings as follows. Let $\boldsymbol{r}=(r_1,r_2,\ldots,r_n)$ be a $n$-tuple of positive integers. Let $\La=\K[x_{ki}\mid 1\le k\le n;\, 1\le i\le r_k]$ be a polynomial algebra in $|\boldsymbol{r}|=r_1+r_2+\cdots+r_n$ variables. Let $G=S_{r_1}\times S_{r_2}\times\cdots \times S_{r_n}$ acting by $\si(x_{ki})=x_{k\si_k(i)}$ for $\si=(\si_1,\si_2,\ldots,\si_n)\in G$ and put $\Ga=\La^G$, $L=\Frac \La$. Let $\SM\cong\Z^{|\boldsymbol{r}|-r_n}$ be the free abelian group with generators $\delta^{ki}$ for $1\le k\le n-1$ and $1\le i\le r_k$ acting on $\La$ by $\delta^{ki}(x_{lj})=x_{lj} - \delta_{kl}\delta_{ij}$. 
Properties \eqref{it:G1}--\eqref{it:G5} are satisfied. Put $\SK=(L\ast\SM)^G$. Consider
\begin{equation}
X_k^\pm = \sum_{i=1}^{r_k} (\delta^{ki})^{\pm 1}\cdot A_{ki}^\pm,\qquad A_{ki}^\pm = \mp 
\frac{\prod_j (x_{k\pm 1,j}-x_{ki})}{\prod_{j\neq i}(x_{kj}-x_{ki})}
\end{equation}
As in \cite{Hartwig2017}, we may identify the OGZ algebra $U(\boldsymbol{r})$ defined in \cite{Mazorchuk1999} with the subring of $\SK$ generated by $\Ga$ and $X_k^\pm$ for $k=1,2,\ldots,n-1$. For any subset $J\subset\{1,2,\ldots,n-1\}$ we define $U_J(\boldsymbol{r})$ to be the corresponding \emph{parabolic subalgebra} generated by $\Ga\cup\{X_k^+\mid 1\le k\le n-1\}\cup\{X_j^-\mid j\in J\}$.

\begin{Theorem}
Let $J\subset\{1,2,\ldots,n-1\}$.
\begin{enumerate}[{\rm (i)}]
\item The parabolic OGZ algebra $U_J(\boldsymbol{r})$ is a Galois $\Ga$-order. 
\item For any character $\xi\in\hat\Ga$ there exists a canonical simple left Gelfand-Zeitlin module $M$ over $U_J(\boldsymbol{r})$ with $M_\xi\neq 0$.
\item If $J=\{1,2,\ldots,n-1\}$ then for any $\xi\in\hat{\Ga}$ there are only finitely many isomorphism classes of simple Gelfand-Zeitlin $U(\boldsymbol{r})$-modules $M$ with $M_\xi\neq 0$.
\end{enumerate}
\end{Theorem}

\begin{proof}
(i): We show that $U_J(\boldsymbol{r})$ is isomorphic to a rational Galois order.
Let $V=\K^N$ and regard $G\subset\GL(V)$ by permutation of coordinates. The relative Jacobian $d_{\sgn}$ associated to the sign character is a product of Vandermonde determinants: $d_{\sgn}=\prod_{k=1}^n\prod_{1\le i<j\le r_k} (x_{ki}-x_{kj})$. Since $X_k^\pm d_{\sgn}  \in \La\ast V$ and $X_k^\pm$ is $G$-invariant for all $k$, the result follows from Theorem \ref{thm:rational-Galois-order}. Note that here $\SM$ is a group iff $J=\{1,2,\ldots,n-1\}$.

\smallskip\noindent
(ii): Immediate by Theorem \ref{thm:rational-Galois-order} and Theorem \ref{thm:principal-nonempty-fiber}. 

\smallskip\noindent
(iii): By part (i), $U(\boldsymbol{r})$ is a Galois $\Ga$-order in $\SK=(L\ast\SM)^G$ where $\SM$ is a group. Thus the claim follows from the main result of \cite{FutOvs2014}.
\end{proof}

\section{Quantum OGZ algebras} \label{sec:qOGZ}

We apply the results from Sections \ref{sec:Galois-orders} and \ref{sec:GZ-modules} to the \emph{quantum OGZ algebras}, introduced in \cite{Hartwig2017}. These are $q$-deformations of the OGZ algebras $U(\boldsymbol{r})$ and include $U_q(\Fgl_n)$ and extended quantized Heisenberg algebras as special cases. To avoid unnecessary clutter, we only treat the standard form rather than the general $(m,p)$-form from \cite{Hartwig2017}.

Let $\K$ be an algebraically closed field of characteristic zero. Let $\boldsymbol{r}=(r_1,r_2,\ldots,r_n)$ be an $n$-tuple of positive integers. Let $q\in\K$ be nonzero and not a root of unity.
Let 
\begin{equation} \label{eq:qOGZ-La}
\Lambda=\K\big[x_{ki}^{\pm 1}\mid 1\le k\le n,\, 1\le i\le r_k\big] 
\end{equation}
be a Laurent polynomial algebra in $|\boldsymbol{r}|=r_1+r_2+\cdots+r_n$ variables, and $L=\Frac\La$.
Let $G$ be the product of complex reflection groups $G(2,2,r_k)$:
\begin{equation} \label{eq:G(m,p,n)-Definition}
G=G_{\boldsymbol{r}} = G_{r_1}\times G_{r_2}\times\cdots \times G_{r_n},\qquad
G_{r_k} = G(2,2,r_k) = S_n\ltimes A(2,2,n),
\end{equation}
\begin{equation}
A(2,2,n)=\big\{\al=(\al_1,\al_2,\ldots,\al_n)\in \{1,-1\}^n\mid \al_1\al_2\cdots \al_n=1\big\}.
\end{equation}
Thus $G_{r_k}$ is isomorphic to the Weyl group of type $D_{r_k}$.
Then $G$ acts naturally on $\Lambda$, hence on $L$, by
\begin{equation}
g(x_{ki})=\al_{ki}x_{k\si_k(i)}
\end{equation}
for 
\begin{equation}\label{eq:g-element}
g=(\si_1\al_1,\si_2\al_2,\ldots,\si_n\al_n)\in G, 
\quad \si_k\in S_{r_k},
\quad \al_k=(\al_{k1},\al_{k2},\ldots,\al_{kr_k})\in A(2,2,r_k).
\end{equation}
Let $\Ga=\La^G$.
Let $\SM\cong\Z^{|\boldsymbol{r}|-r_n}$, written multiplicatively, with $\Z$-basis $\big\{\delta^{ki}\mid 1\le k\le n-1,\, 1\le i\le r_k\big\}$, acting on $\La$ by $\delta^{ki}(x_{lj})=q^{-\delta_{kl}\delta_{ij}}x_{lj}$.
For simplicity we use the Galois ring realization obtained in \cite{Hartwig2017} as the definition of these algebras. 

\begin{Definition} \label{def:OGZ}
The \emph{quantum OGZ algebra of signature $\boldsymbol{r}$}, denoted $U_q(\boldsymbol{r})$, is the subring of $\SK$ generated by $\Ga$ and $X_k^\pm$ for $k\in\iv{1}{n-1}$, where
\begin{equation}
X_k^{\pm} = \sum_{i=1}^{r_k} (\delta^{ki})^{\pm 1}\cdot A_{ki}^{\pm} \in L\ast \SM,
\end{equation}
where, with the convention $r_0=0$,
\begin{equation}\label{eq:A_ki-Definition}
A_{ki}^{\pm}=\mp x_{ki}^{-(r_{k\pm 1}-r_k)}\cdot\frac{\displaystyle\prod_{j=1}^{r_{k\pm 1}} \big( (x_{k\pm 1,j}/x_{ki})-(x_{k\pm 1,j}/x_{ki})^{-1} \big) /(q-q^{-1})}{\displaystyle\prod_{\substack{1\le j\le r_k\\ j\neq i}}
 \big( (x_{kj}/x_{ki}) - (x_{kj}/x_{ki})^{-1}\big)/(q-q^{-1})}.
\end{equation}
For $J\subset\{1,2,\ldots,n-1\}$ define the \emph{parabolic quantum OGZ algebra} $U_q(\boldsymbol{r};J)$ to be the subalgebra of $U_q(\boldsymbol{r})$ generated by $\Ga\cup \{X_k^+\}_{k=1}^n\cup\{X_j^-\}_{j\in J}$.
\end{Definition}

\begin{Example}
$U_q(1,2,\ldots,n)\cong U_q(\mathfrak{gl}_n)$. For details, see \cite{FutHar2014,Hartwig2017}.
\end{Example}

\begin{Example}
$U_q(1,1)$ is isomorphic to $\mathcal{H}_q[t,t^{-1}]$ where $\mathcal{H}_q$ is the \emph{quantized Heisenberg algebra} with generators $X,Y,K^{\pm 1}$ and relations
\begin{gather}
KK^{-1}=K^{-1}K=1,\\
YX=\frac{K-K^{-1}}{q-q^{-1}},\qquad 
XY=\frac{qK-q^{-1}K^{-1}}{q-q^{-1}},\\
KXK^{-1}=qX, \qquad 
KYK^{-1}=q^{-1}Y.
\end{gather}
\end{Example}

\begin{Definition}
The \emph{$q$-Vandermonde of signature $\boldsymbol{r}$} is the element of $\La$ given by
\begin{equation}
\mathcal{V}_q(\boldsymbol{r})=\prod_{k=1}^n \;
\prod_{1\le i<j\le r_k} \frac{(x_{ki}/x_{kj})-(x_{ki}/x_{kj})^{-1}}{q-q^{-1}}.
\end{equation}
\end{Definition}

\begin{Lemma}\label{lem:q-Vandermonde}
Let $\sgn:G\to\{\pm 1\}$ be the linear character given by
\[
\sgn:(\si_1\al_1,\si_2\al_2,\ldots,\si_n\al_n)\mapsto (\sgn \si_1)(\sgn \si_2)\cdots(\sgn \si_n).
\]
Then
\begin{equation}
\Lambda^G_{\sgn} = \Ga \cdot \mathcal{V}_q(\boldsymbol{r}).
\end{equation}
\end{Lemma}

\begin{proof}
Let $f\in\La^G_{\sgn}$. Since
\begin{equation}
\frac{(x_{ki}/x_{kj})-(x_{ki}/x_{kj})^{-1}}{q-q^{-1}}=
\frac{(x_{ki}x_{kj})^{-1}}{q-q^{-1}}(x_{ki}^2-x_{kj}^2)
\end{equation}
it suffices to show that $f$ is divisible by
\begin{equation}\label{eq:m-vandermonde}
\prod_{k=1}^n \prod_{1\le i<j\le r_k} (x_{ki}^2-x_{kj}^2).
\end{equation}
Since $f$ is $G$-alternating, it is fixed by $A_{\boldsymbol{r}} = A(2,2,r_1)\times A(2,2,r_2)\times\cdots \times A(2,2,r_n)$. Thus
\begin{equation}
f=\sum_{k=(k_1,k_2,\ldots,k_n)\in\{0,1\}^n} f_k\cdot (x_{11}x_{12}\cdots x_{1r_1})^{k_1}
(x_{21}x_{22}\cdots x_{2r_2})^{k_2}\cdots 
(x_{n1}x_{n2}\cdots x_{nr_n})^{k_n},
\end{equation}
where $f_k$ are $S_{\boldsymbol{r}}$-alternating Laurent polynomials in the variables $\{x_{ki}^2\mid 1\le k\le n,\, 1\le i\le r_k\}$, hence divisible by \eqref{eq:m-vandermonde}.
\end{proof}

\begin{Theorem}\label{thm:qOGZ}
Let $J\subset\{1,2,\ldots,n-1\}$ and $U=U_q(\boldsymbol{r};J)$ be a parabolic quantum OGZ algebra and $\Ga$ be its Gelfand-Zeitlin subalgebra.
\begin{enumerate}[{\rm (i)}]
\item $U$ is a Galois $\Ga$-order in $\SK$.
\item $\Ga$ is maximal commutative in $U$.
\item For any $\xi\in\hat{\Ga}$ there exists a canonical simple left Gelfand-Zeitlin $U$-module $M$ with $M_\xi\neq 0$.
\item If $J=\{1,2,\ldots,n\}$ then for any $\xi\in\hat{\Ga}$ there are only finitely many isomorphism classes of simple Gelfand-Zeitlin $U$-modules $M$ with $M_\xi\neq 0$.
\end{enumerate}
\end{Theorem}

\begin{proof}
(i): Let $\SM_J$ be the submonoid of $\SM$ generated by $\delta^{ka}$ and $(\delta^{jb})^{-1}$ for $k\in\iv{1}{n-1}$, $a\in\iv{1}{r_k}$, $j\in J$ and $b\in\iv{1}{r_j}$.
We show that $U$ is a co-principal Galois $\Ga$-order in $\SK_J=(L\ast(\SM_J)\big)^G$.
It is easy to verify that conditions \eqref{it:G1}--\eqref{it:G5} hold for $\La$, $G$ and $\SM_J$ as defined in \eqref{eq:qOGZ-La}, \eqref{eq:G(m,p,n)-Definition}.
By Corollary \ref{cor:anti-criterion}, it suffices to check that $(X_k^+)^\dagger, (X_j^-)^\dagger\in \SK_\Ga$ for all $k\in\iv{1}{n-1}$ and $j\in J$.
Observe that multiplying $(X_k^\pm)^\dagger$ by the $q$-Vandermonde from the left clears all denominators. So the claim is immediate by Lemma \ref{lem:q-Vandermonde} and Lemma \ref{lem:co-principal-test}.

\smallskip\noindent
(ii): By part (i) and Proposition \ref{prp:GO-Gamma-maximal-commutative}.

\smallskip\noindent
(iii): By part (i), Remark \ref{rem:dagger} and Theorem \ref{thm:principal-nonempty-fiber}.

\smallskip\noindent
(iv): In this case $\SM_J=\SM$ is which is a group so the statement follows from part (i) and the main result of \cite{FutOvs2014}.
\end{proof}


\begin{thebibliography}{XXXXX}

\bibitem[B10]{Broue}
  {\sc M. Brou{\'e}},
  {\em Introduction to complex reflection groups and their braid groups},
  Lecture Notes in Mathematics 1988, Springer-Verlag, 2010.

\bibitem[BK06]{BruKle2006}
  {\sc J. Brundan, A. Kleshchev},
  {\em Shifted Yangians and finite W-algebras},
  Adv. Math. 200 (2006) 136--195.

\bibitem[CE14]{ColEve2014}
  {\sc M. Colarusso, S. Evens},
  {\em The Gelfand–Zeitlin integrable system and K-orbits on the flag variety}.
  In: Howe R., Hunziker M., Willenbring J. (eds) Symmetry: Representation Theory and Its Applications. Progress in Mathematics, vol 257. Birkh\"{a}user, New York, NY, 2014.

\bibitem[DFO94]{DroFutOvs1994}
  {\sc Yu.A. Drozd, V.M. Futorny, S.A. Ovsienko},
  {\em Harish-Chandra subalgebras and Gelfand-Zetlin modules}.
  In: Finite-dimensional algebras and related topics, 79--93,
  Kluwer, 1994.

\bibitem[EMV17]{EarMazVis2017}
  {\sc N. Early, V. Mazorchuk, E. Vishnyakova},
  {\em Canonical Gelfand-Zeitlin modules over orthogonal Gelfand-Zeitlin algebras},
  arXiv:1709.01553.

\bibitem[FGR14]{FutGraRam2014}
	{\sc V. Futorny, D. Grantcharov, L.E. Ram\'{i}rez},
	{\em On the classification of irreducible {G}elfand-{T}setlin modules of {$\mathfrak{sl}(3)$}}.
	In: Recent advances in representation theory, quantum groups, algebraic geometry, and related topics,
	Contemp. Math. 623, 	63--79,
	Amer. Math. Soc., Providence, RI, 2014.

\bibitem[FGR15]{FutGraRam2015}
	{\sc V. Futorny, D. Grantcharov, L.E. Ram\'{i}rez},
	{\em Irreducible generic Gelfand-Tsetlin modules of $\mathfrak{gl}(n)$},
	Symmetry, Integrability and Geometry: Methods and Applications 11 (2015) 1--13.
	
\bibitem[FGR16]{FutGraRam2016}
	{\em Singular {G}elfand-{T}setlin modules of {$\mathfrak{gl}(n)$}},
	{\sc V. Futorny, D. Grantcharov, L.E. Ram\'{i}rez},
	Adv. Math. 290 (2016) 453--482.

\bibitem[FGR17]{FutGraRam2017}
	{\sc V. Futorny, D. Grantcharov, L.E. Ram\'{i}rez},
	{\em New singular {G}elfand-{T}setlin {$\mathfrak{gl}(n)$}-modules of index 2},
	Comm. Math. Phys. 355 no. 3 (2017) 1209--1241.

\bibitem[FMO10a]{FutMolOvs2010a}
  {\sc V. Futorny, A. Molev, S. Ovsienko},
  {\em Gelfand-Tsetlin bases for representations of finite W-algebras and shifted
Yangians},
  in: Lie Theory and its Applications in Physics VII, Varna, Bulgaria, June 2007, Heron Press, Sofia, 2008, pp. 352--363.

\bibitem[FMO10b]{FutMolOvs2010}
  {\sc V. Futorny, A. Molev, S. Ovsienko},
  {\em The Gelfand-Kirillov conjecture and Gelfand-Tsetlin modules for finite $W$-algebras},
  Adv. Math. 223 (2010) 773--796.
  
\bibitem[FO10]{FutOvs2010}
  {\sc V. Futorny, S. Ovsienko},
  {\em Galois orders in skew monoid rings},
  J. Algebra 324 (2010) 598--630.

\bibitem[FH14]{FutHar2014}
  {\sc V. Futorny, J.T. Hartwig},
  {\em Solution of a $q$-difference Noether problem and the quantum Gelfand-Kirillov conjecture for $\Fgl_n$},
  Math. Z. 276 (2014) 1--37.

\bibitem[FO14]{FutOvs2014}
  {\sc V. Futorny, S. Ovsienko},
  {\em Fibers of characters in Gelfand-Tsetlin categories},
  Trans. Ameri. Math. Soc. 366 no. 8 (2014) 4173--4208.

\bibitem[FOS11]{FutOvsSao2011}
  {\sc V. Futorny, S. Ovsienko, M. Saor\'\i n},
  {\em Torsion theories induced from commutative subalgebras},
  J. Pure Appl. Algebra 215 no. 12 (2011) 2937--2948.

\bibitem[FGR16]{FutGraRam2016}
  {\sc V. Futorny, D. Grantcharov, L.-E. Ram\'{i}rez},
  {\em Singular Gelfand-Tsetlin modules for $\Fgl(n)$},
  Adv. Math. 290 (2016) 453--482.

\bibitem[GG02]{GanGin2002},
  {\sc W.L. Gan, V. Ginzburg},
  {\em Quantization of Slodowy slices},
  International Mathematics Research Notices 2002(5), 243--255.
  
\bibitem[GR16]{GomRam2016}
  {\sc C.A. Gomes, L.E. Ramirez},
  {\em Families of irreducible singular Gelfand-Tsetlin modules of $\mathfrak{gl}(n)$},
  arXiv:1612.00636.  

\bibitem[GZ50a]{GelZei1950a}
  {\sc I.M. Gelfand, M.L. Tsetlin},
  {\em Finite-dimensional representations of the group of unimodular representations},
  Dokl. Akad. Nauk SSSR (N.S.) 71 (1950) 825--282.

\bibitem[GZ50b]{GelZei1950b}
  {\sc I.M. Gelfand, M.L. Tsetlin},
  {\em Finite-dimensional representations of the group of orthogonal matrices}, 
  Dokl. Akad. Nauk SSSR (N.S.) 71 (1950) 1017--1020.
  
\bibitem[H17]{Hartwig2017}
  {\sc J.T. Hartwig},
  {\em The $q$-difference Noether problem for complex reflection groups and quantum OGZ algebras},
  Comm. Alg. 45 no. 3 (2017) 1166--1176.

\bibitem[J85]{Jacobson1}
  {\sc N. Jacobson},
  {\em Basic algebra I},
  Dover, 2009.
  Originally published: 2nd ed. San Fransisco: W. H. Freeman 1985.

\bibitem[KW06a]{KosWal2006a}
  {\sc B. Kostant, N. Wallach},
  {\em Gelfand-Zeitlin theory from the perspective of classical mechanics I}.
  In: Studies in Lie Theory Dedicated to A. Joseph on his Sixtieth Birthday, Progress in Mathematics, 243 (2006) 319--364.

\bibitem[KW06b]{KosWal2006b}
  {\sc B. Kostant, N. Wallach},
  {\em Gelfand-Zeitlin theory from the perspective of classical mechanics II}.
  In: The Unity of Mathematics In Honor of the Ninetieth Birthday of I. M. Gelfand, Progress in Mathematics, 244 (2006) 387--420.

\bibitem[KS97]{KliSch1997}
  {\sc A. Klimyk, K. Schm\"{u}dgen},
  {\em Quantum groups and their representations},
  Texts and Monographs in Physics. Springer-Verlag, Berlin, 1997.

\bibitem[M99]{Mazorchuk1999}
  {\sc V. Mazorchuk},
  {\em Orthogonal Gelfand-Zetlin algebras, I},
  Beitr. Algebra Geom. 40 no. 2 (1999) 399--415.

\bibitem[MT00]{MazTur2000}
  {\sc V. Mazorchuk, L. Turowska},
  {\em On Gelfand-Zetlin modules over $U_q(gl_n)$},
  Czech. J. Phys. 50 no. 1 (2000) 139--144.

\bibitem[MPT03]{MazPonTur2003}
  {\sc V. Mazorchuk, M. Ponomarenko, L. Turowska, L.},
  {\em Some associative algebras related to {$U(\mathfrak{g})$} and twisted generalized {W}eyl algebras},
  Math. Scand. 92 no. 1 (2003) 5--30.

\bibitem[M86]{Matsumura}
  {\sc H. Matsumura},
  {\em Commutative ring theory},
  Cambridge studies in advanced mathematics 8,
  Cambridge University Press, 1986.

\bibitem[O03]{Ovsienko2003}
  {\sc S. Ovsienko},
  {\em Strongly nilpotent matrices and Gelfand-Zetlin modules}.
  In: Special issue on linear algebras methods in representation theory. Linear algebra Appl. 365 (2003) 349--367.

\bibitem[P02]{Premet2002}
  {\sc A. Premet},
  {\em Special transverse slices and their enveloping algebras},
  Adv. Math. 170 (2002) 1--55.

\bibitem[RZ17]{RamZad2017}
  {\sc L.R. Ram\'{i}rez, P. Zadunaisky},
  {\em Gelfand-Tsetlin modules over $\mathfrak{gl}(n,\mathbb{C})$ with arbitrary characters},
  arXiv:1705.10731.

\bibitem[RS99]{RagSor1999}
  {\sc E. Ragoucy, P. Sorba},
  {\em Yangian realisations from finite W-algebras},
  Commun. Math. Phys. 203 (1999) 551--572.

\bibitem[T89]{Terao1989}
  {\sc H. Terao},
  {\em The Jacobians and the discriminants of finite reflection groups},
  T\^{o}hoku Math. J. 41 (1989) 237--247.

\bibitem[V17a]{Vishnyakova2017a}
  {\sc E. Vishnyakova},
  {\em A geometric approach to $1$-singular Gelfand-Tsetlin $\Fgl_n$-modules},
  arXiv:1704.00170.
  
\bibitem[V17b]{Vishnyakova2017b}
  {\sc E. Vishnyakova},
  {\em Geometric approach to $p$-singular Gelfand-Tsetlin $\Fgl_n$-modules},
  arXiv:1705.05793.

\bibitem[Za17]{Zadunaisky2017}
  {\sc P. Zadunaisky},
  {\em A new way to construct $1$-singular Gelfand-Tsetlin modules},
  Algebra Discrete Math. 23 no. 1 (2017) 180--193.

\bibitem[Zhe73]{Zhelobenko1973}
  {\sc D.P. Zhelobenko},
  {\em Compact Lie groups and their representations},
  Transl. of Math. Monographs 40 AMS, Providence RI, 1973.

\end{thebibliography}
\end{document}